\documentclass[12pt,final]{amsart}
\usepackage{amssymb,amsthm,amsmath}
\usepackage{graphicx}
\usepackage{fullpage}

\usepackage{mathtools}
\usepackage[colorlinks = true, urlcolor={blue}, citecolor={blue}]{hyperref}
\usepackage{float}

\usepackage{tikz,tikz-cd}

\newcommand{\var}{\operatorname{Var}}

\newcommand{\tmix}{t_{\operatorname{mix}}(\varepsilon)}

\newtheorem{theorem}{Theorem}[section]
\newtheorem{proposition}[theorem]{Proposition}
\newtheorem{lemma}[theorem]{Lemma}
\newtheorem{remark}[theorem]{Remark}

\newcommand{\Exp}[2]{\mathbb{E}_{#1}\left[#2\right]}

\newcommand{\Ent}{\operatorname{Ent}}
\newcommand{\Ecal}{\mathcal{E}}
\newcommand{\R}{\mathbb{R}}

\DeclarePairedDelimiterXPP{\nrm}[2]{}{\lVert}{\rVert}{\ensuremath{_{#1}}}{\ifblank{#2}{\:\cdot\:}{#2}}

\newcommand{\TV}[1]{\nrm*{TV}{#1}}

\title{Log-Sobolev under random monotone censoring}
\author[P. O. Santos, R. Tripathi, and P. Youssef]{Patrick Oliveira Santos \and Raghavendra Tripathi \and Pierre Youssef}

\address{Raghavendra Tripathi. Division of Science, NYU Abu Dhabi, Abu Dhabi, UAE}
\email{r.tripathi@nyu.edu}
\address{Patrick Oliveira Santos. Division of Science, NYU Abu Dhabi, Abu Dhabi, UAE}
\email{po2150@nyu.edu}
\address{Pierre Youssef. Center for Interdisciplinary Data Science and AI, NYUAD Research Institute, UAE \& 
Division of Science, NYU Abu Dhabi, Abu Dhabi, UAE \& 
Courant Institute of Mathematical Sciences, New York University, 251 Mercer st, New York,
NY 10012, USA.}
\email{yp27@nyu.edu}

\begin{document}

\begin{abstract}
We show that the logarithmic Sobolev inequality of the Boolean cube is stable under random monotone censoring. More precisely, if $A_n\subseteq \{0,1\}^n$ is chosen uniformly among all monotone subsets, then the logarithmic Sobolev constant of the censored walk on $A_n$ is of order $n$ with high probability. 
As a consequence, several analytic and probabilistic properties of the Boolean cube persist for a typical monotone subset: the censored semigroup is hypercontractive, the uniform measure on $A_n$  satisfies Gaussian concentration for Lipschitz observables, and the associated walk mixes in time $O(n\log n)$. The latter proves a conjectured mixing bound of Ding and Mossel for almost all monotone sets. 
The result is genuinely typical rather than universal. We construct monotone sets of density bounded away from zero whose logarithmic Sobolev constant is of order $n^2$. 

To prove the result, we establish a sharp logarithmic Sobolev inequality for Hamming caps and combine it with a harmonic extension argument transferring this inequality to monotone sets lying between nearby caps, together with a structural theorem of Korshunov on random monotone sets.

\end{abstract}

\maketitle
\section{Introduction}
\label{sec:intro}

The Boolean cube $\Omega_n=\{0,1\}^n$ is one of the fundamental objects in probability, combinatorics, and theoretical computer science. Its geometry is reflected in a collection of remarkable analytic phenomena, including hypercontractivity \cite{Bon, Beck, Gross}, logarithmic Sobolev inequalities \cite{Gross, Dia96LSI}, concentration of measure (see \cite{Led99Concentration, BLM}), and rapid mixing of the simple random walk (see \cite{Dia96LSI, S-C97}). 

Recall that $\Omega_n$ carries the coordinatewise partial order 
$$
x\preceq y\qquad \Longleftrightarrow \qquad x_i\le y_i,\quad 1\le i\le n.
$$
A subset $A\subseteq \Omega_n$ is called \emph{monotone} (or increasing) if $x\in A$ and $x\preceq y$ imply that $y\in A$. Equivalently, monotone subsets correspond to monotone Boolean functions. They form one of the most important and extensively studied classes of Boolean functions \cite{ODon14}. 
Monotone sets arise naturally in percolation \cite{BR06, Grim89} as monotone events, connectivity properties of random graphs \cite{Bollobas}, and statistical physics \cite{FKG, G06}. They are also closely connected to a rich theory that includes property testing, influences, and sharp thresholds~\cite{FK96, KKL88, KS06, GGLRS}. 

A natural question is whether monotone subsets retain the analytic properties of the ambient cube. More precisely, if one conditions the cube on a monotone event, do logarithmic Sobolev inequalities, concentration estimates, hypercontractivity, and mixing bounds remain essentially unchanged? 

In this paper, we study this question through the logarithmic Sobolev inequality. A natural Markov chain on $\Omega_n$ is the lazy simple random walk, which, at each step, chooses a coordinate uniformly at random and resamples its value. Given a monotone subset $A\subseteq \Omega_n$, the associated censored walk is obtained by suppressing all transitions from $A$ to $A^c$.  Equivalently, the walk evolves as the lazy simple random walk on the induced subgraph of $\Omega_n$ with vertex set $A$. The resulting Markov chain is reversible with respect to the uniform measure on $A$. 

Among the functional inequalities on the cube, the logarithmic Sobolev inequality occupies a distinguished position. It is equivalent to the hypercontractivity of the associated semigroup \cite{Gross} and implies both sub-Gaussian concentration inequalities \cite{Led99Concentration, BLM} and quantitative mixing estimates \cite{Dia96LSI, S-C97}. We write $t_{\operatorname{LS}}(A)$ for the smallest constant such that 
$$
\Ent_{\pi_A}(f^2)\leq t_{\operatorname{LS}}(A)\, \Ecal_{A}(f,f),
$$
for every function $f:\, A\to \R$, where $\pi_A$ denotes the uniform measure on $A$, and $\Ecal$ is the Dirichlet form of the censored walk. We refer to Section~\ref{sec:preliminaries} for background and conventions. For the full cube, the Bonami--Beckner--Gross~\cite{Bon,Beck,Gross} log-Sobolev inequality stipulates (in our normalization) that $t_{\operatorname{LS}}(\Omega_n)\asymp n$ (also see~\cite[Example 3.6]{salez2025modern}).

Ding and Mossel~\cite{ding2014mixing} initiated the study of the mixing time of the censored walk to a monotone set. They conjectured that for any monotone set $A\subseteq \Omega_n$ with $|A|\geq \alpha 2^{n}$ for some constant $\alpha \in (0, 1]$, the mixing time of the censored walk is bounded by $\kappa_\alpha\cdot n\log n$, where $\kappa_\alpha $ is a constant depending on $\alpha$ (but independent of $n$). We should recall that the mixing time of the random walk on $\Omega_n$ (equivalently, when $\alpha =1$) is known very precisely. Thus, the essence of Ding and Mossel's conjecture is that \emph{monotone censoring does not worsen the mixing time estimates in any significant way}. We refer the reader to \cite[Section 5.3.2]{fathi2019quelques} for a more general context in which this problem can be framed. We also refer the reader to~\cite{FK13, Hol11} for the effect of censoring on mixing properties. 
Ding and Mossel~\cite{ding2014mixing} used ideas from property testing to derive a sharp bound for the conductance, which, in turn, implies that the mixing time is $O_{\alpha }(n^3)$ for any monotone set $A$ with $|A|\geq \alpha 2^{n}$. Recently, Fei and Pinto Jr~\cite{fei2025spectral} (see also ~\cite{chang2025poincar} for an alternate proof) established the sharp Poincar\'e constant of order $O(n)$ for the censored Markov chain, improving the mixing time estimate to $O(n^2)$.  
While the results in \cite{fei2025spectral, chang2025poincar} identify the correct relaxation scale of the censored walk, Poincar\'e inequality inherently loses an additional factor $n$ when converted into a bound on the mixing time. In contrast, log-Sobolev inequality provides a sharper control (see \eqref{eq:LSI-mixing}) and an order-$n$ logarithmic Sobolev inequality would recover the conjectured mixing bound while simultaneously implying hypercontractivity and Gaussian concentration for the uniform measure on $A$. 

This naturally raises the question of whether a logarithmic Sobolev inequality comparable to that of the full cube persists under monotone censoring. 
Our main result shows that this is indeed the case for a typical monotone subset. 

\begin{theorem}
\label{thm:highprobLSI}
    There exist universal constants $c, C>0$ such that the following holds. Let $A_n$ be a uniformly chosen random monotone subset of $\Omega_n$. Then
    \[
    \mathbb{P}\left(c\cdot n\leq  t_{\operatorname{LS}}(A_n)\leq C\cdot n\right) = 1-o_n(1).
    \]
\end{theorem}

Theorem~\ref{thm:highprobLSI} shows that, for a typical monotone subset, the logarithmic Sobolev inequality of the cube survives monotone censoring. Standard consequences of logarithmic Sobolev inequalities therefore extend to almost every monotone set. In particular, with probability $1-o_n(1)$, the censored semigroup on $A_n$ is hypercontractive, and the uniform measure on $A_n$ satisfies Gaussian concentration for Lipschitz observables. More precisely, there exists a universal constant $c>0$ such that, for every function $f:\, A_n\to \R$ which is 1-Lipschitz with respect to the Hamming distance,  
$$
\pi_{A_n}\big(\vert f-\Exp{\pi_{A_n}}{f}\vert \geq t\big) \leq 2e^{-\frac{ct^2}{n}},\qquad t\geq 0.
$$
Moreover, the associated walk mixes in time $O(n\log n)$; see \eqref{eq:LSI-mixing}. Thus, the conjectured Ding--Mossel \cite{ding2014mixing} mixing bound holds for a uniformly random monotone set.

The proof of Theorem~\ref{thm:highprobLSI} rests on two deterministic ingredients which may be of independent interest. First, we establish a sharp logarithmic Sobolev inequality for Hamming caps (see Theorem~\ref{thm:hamming-caps}). For any $0\leq s\leq n-1$ and $H_s=\{x\in\Omega_n:\ |x|\ge s\}$, we prove 
$$
t_{\operatorname{LS}}(H_s)
\asymp
n\log\frac{en}{n-s+1}.
$$
In particular, caps around the middle layer have the same order-$n$ logarithmic Sobolev constant as the full cube, while smaller caps exhibit the expected logarithmic deterioration. Second, we prove a comparison theorem showing that logarithmic Sobolev inequalities may be transferred from a Hamming cap to any monotone set trapped between two nearby caps (see Theorem~\ref{thm:sandwich-monotone-set}). The proof is based on a harmonic extension argument that fills in the missing layers one at a time while controlling the growth of the Dirichlet energy.
Finally, we invoke a theorem of Korshunov \cite{Korshunov03, Korshunov81} asserting that a uniformly chosen monotone subset differs from a suitable Hamming cap only within a bounded-width slab around the middle layer.

Theorem~\ref{thm:highprobLSI} should be read together with a worst-case result proved in Section~\ref{sec:large-LSI} where we show that for any $\alpha\in (0, 1/2]$, we have
$$
     \sup_{\substack{A\subseteq \Omega_n \text{ monotone}\\ |A|\geq \alpha |\Omega_n|} }t_{\operatorname{LS}}(A)  \asymp_\alpha  n^2.
    $$
This implies that the log-Sobolev analogue of the Ding–Mossel conjecture is false in full generality: there are monotone sets of density bounded away from zero whose log-Sobolev constant is of order $n^2$. Thus, Theorem~\ref{thm:highprobLSI} identifies a genuine typical-case phenomenon. Although an arbitrary dense monotone set can contain exponentially small traps that destroy the cube-scale log-Sobolev inequality, a uniformly random monotone set has no such obstruction; its log-Sobolev behavior is governed instead by the middle layers of the cube.

The paper is organized as follows. In Section~\ref{sec:preliminaries}, we set up the notations and formally introduce the definitions and structural results for random monotone subsets. Section~\ref{sec:LSI-caps} establishes a sharp logarithmic Sobolev inequality for Hamming caps. In Section~\ref{sec:LSI-sandwich} we prove a logarithmic Sobolev inequality for monotone sets lying between nearby caps and conclude the proof of Theorem~\ref{thm:highprobLSI}. Finally, in Section~\ref{sec:large-LSI} we investigate monotone subsets with a large logarithmic Sobolev constant.

\section{Preliminaries}\label{sec:preliminaries}

Throughout the paper, $|A|$ denotes the cardinality of a finite set $A$. We write $o_n(1)$ for a quantity tending to zero as $n\to\infty$. Given two sequences $a_n,b_n>0$ and $\alpha \in \R$, we write $a_n\asymp_\alpha b_n$ if there exist universal constants $c_{\alpha}, C_{\alpha}>0$ that depend only on $\alpha$ such that $c_{\alpha}b_n\leq a_n\leq C_{\alpha}b_n$ for all sufficiently large $n$. We denote $a_n \asymp b_n$ if $\alpha=1$, that is, whenever the constants $c, C>0$ above are universal.

Let $\Omega$ be a finite state space equipped with a probability measure $\pi$. We write $\Exp{\pi}{f}=\sum_{x\in\Omega} \pi(x) f(x)$ for the expectation,  
$$
\langle f,g\rangle_\pi= \Exp{\pi}{fg}
$$
for the associated inner product, and
$$
\Vert f\Vert_{L^2(\pi)}=\sqrt{\langle f,f\rangle_{\pi}}
$$
for the corresponding $L^2(\pi)$-norm. 

Let $P$ be a Markov kernel on $\Omega$ which is reversible with respect to $\pi$, namely 
$$
\pi(x)P(x,y)=\pi(y)P(y,x),\qquad x,y\in \Omega.
$$
The associated Dirichlet form is 
$$
\Ecal_P(f,g)=\langle f, (I-P)g\rangle_\pi.
$$
Equivalently, 
$$
\Ecal_P(f,g)=\frac12\sum_{x,y\in \Omega} \pi(x)P(x,y) (f(x)-f(y))(g(x)-g(y)).
$$
For a non-negative function $f:\, \Omega \to \R$, define the entropy 
$$
\Ent_{\pi}[f]= \Exp{\pi}{f\log f} - \Exp{\pi}f\log\Exp{\pi}{f}.
$$
The logarithmic Sobolev constant $t_{\operatorname{LS}}(P)$ is the smallest constant such that 
$$
\Ent_{\pi}[f^2] \leq t_{\operatorname{LS}}(P) \cdot \Ecal_P(f, f)
$$
for all $f:\, \Omega\to \R$. 
For any $\varepsilon>0$, the $\varepsilon$-\emph{mixing time} (or simply the mixing time) is defined as 
$$
   \tmix =\min \left\{k\geq 0: \sup_{\mu \in \mathcal{P}}\TV{\mu P^k-\pi}\leq \varepsilon\right\}, 
$$
where $\TV{\cdot}$ denotes the total variation norm and $\mathcal{P}$ denotes the set of all probability measures on $\Omega$. 

We now specialize to the Boolean cube $\Omega_n=\{0,1\}^n$. 
For $x=(x_1,\ldots,x_n)\in\Omega_n$, we write
$$
|x|=\sum_{i=1}^n x_i
$$
for its Hamming weight. For $x,y\in \Omega_n$, we write $x\vee y\in \Omega_n$ for the entrywise maximum 
\begin{align*}
    x\vee y=(\max\{x_1,y_1\},\ldots, \max\{x_n,y_n\}).
\end{align*}
For every $0\le k\le n$, let
$$
E_k=\{x\in\Omega_n:\ |x|=k\}
$$
denote the $k$-th Hamming layer, and
$$
H_s=\bigcup_{k\ge s}E_k
$$
the corresponding Hamming cap.  
For $x,y\in\Omega_n$, we write $x\sim y$ if $x$ and $y$ differ in exactly one coordinate. 

Let $A\subseteq\Omega_n$ be a monotone set. We write $\pi_A$ for the uniform measure on $A$. We use the convention of Ding and Mossel~\cite{ding2014mixing} to define the transition kernel of the censored walk on $A$ as
$$
P_A(x,y)=
\begin{cases}
\frac{1}{2n},& x\sim y,\ y\in A,\\
0,& x\neq y,\ y\notin A,
\end{cases}
$$
with the diagonal $P_A(x, x)$ chosen so that rows sum to one. Note that when $A=\Omega_n$, this is the lazy random walk on the cube.  
We write
$
\Ecal_A:=\Ecal_{P_A}
$
for the associated Dirichlet form. Explicitly,
$$
\Ecal_A(f,g)
=
\frac{1}{4n|A|}
\sum_{\substack{x,y\in A\\ x\sim y}}
(f(x)-f(y))(g(x)-g(y)).
$$
Similarly, we denote $
t_{\operatorname{LS}}(A):=
t_{\operatorname{LS}}(P_A).
$
A standard consequence of the logarithmic Sobolev inequality~\cite[Corollary 3.2]{salez2025modern} is
\begin{equation}\label{eq:LSI-mixing}
    \tmix \leq \frac{ t_{\operatorname{LS}}(A)}{4}\left(\log\log|A| + \log \frac{1}{2\varepsilon^2}\right). 
\end{equation}

Our proof relies on a structural property concerning random monotone subsets of the cube. Let
\[
m=\Big\lfloor \frac n2\Big\rfloor.
\]
Let $A_n$ be a uniform random monotone subset of $\Omega_n$. It follows from~\cite[Section 1.2]{Korshunov03} that there is a universal constant $K$ ($K$ can be taken to be $3$) such that 
\begin{equation}
\label{eqn:TypicalA}
    \mathbb{P}(H_{m+K}^{(n)}\subseteq A_n\subseteq H_{m-K}^{(n)})=1-o_n(1)\;.
\end{equation}
See Figure~\ref{fig:RandomMonotoneSet} for the illustration. 

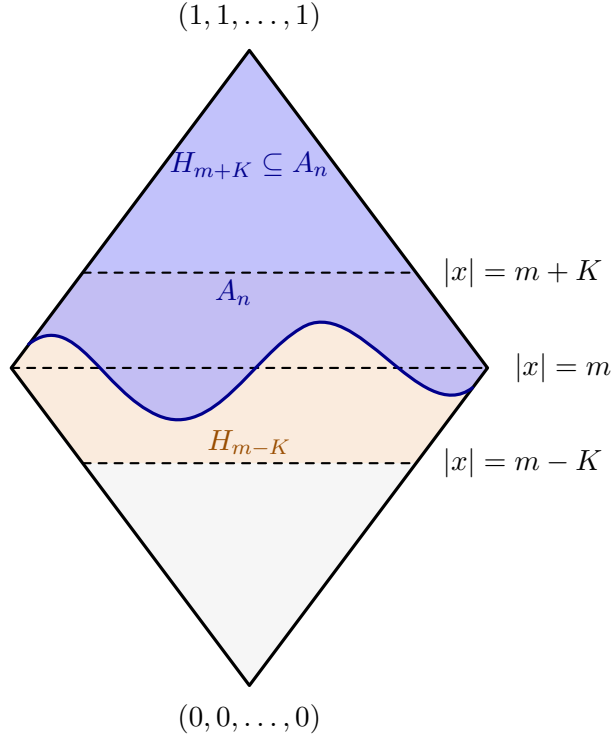
\begin{figure}[H]
\centering
\begin{tikzpicture}[
    scale=1.05,
    line cap=round,
    line join=round,
    >=Latex,
    plane/.style={thick, dashed},
    boundary/.style={very thick, blue!55!black},
    labelstyle/.style={font=\small}
]

\coordinate (Top)    at (0,4);
\coordinate (Bottom) at (0,-4);
\coordinate (Left)   at (-3,0);
\coordinate (Right)  at (3,0);

\coordinate (Uleft)  at (-2.1,1.2);
\coordinate (Uright) at (2.1,1.2);

\coordinate (Mleft)  at (-3,0);
\coordinate (Mright) at (3,0);

\coordinate (Lleft)  at (-2.1,-1.2);
\coordinate (Lright) at (2.1,-1.2);

\coordinate (Gleft)  at (-2.78,0.30);
\coordinate (Gright) at (2.81,-0.25);

\path[fill=gray!7] (Top) -- (Right) -- (Bottom) -- (Left) -- cycle;

\path[fill=orange!25, opacity=0.45]
    (Uleft) -- (Uright) -- (Right) -- (Lright) --
    (Lleft) -- (Left) -- cycle;

\path[fill=blue!35, opacity=0.65]
    (Top) -- (Right) -- (Gright)
    .. controls (2.15,-0.75) and (1.35,0.95) .. (0.65,0.50)
    .. controls (0.05,0.15) and (-0.45,-0.90) .. (-1.15,-0.60)
    .. controls (-1.85,-0.30) and (-2.20,0.75) .. (Gleft)
    -- cycle;

\draw[very thick] (Top) -- (Right) -- (Bottom) -- (Left) -- cycle;

\draw[plane] (Uleft) -- (Uright);
\draw[plane] (Mleft) -- (Mright);
\draw[plane] (Lleft) -- (Lright);

\draw[boundary]
    (Gright)
    .. controls (2.15,-0.75) and (1.35,0.95) .. (0.65,0.50)
    .. controls (0.05,0.15) and (-0.45,-0.90) .. (-1.15,-0.60)
    .. controls (-1.85,-0.30) and (-2.20,0.75) .. (Gleft);

\node[labelstyle, above=3pt] at (Top) {$(1,1,\ldots,1)$};
\node[labelstyle, below=3pt] at (Bottom) {$(0,0,\ldots,0)$};

\node[labelstyle, right=6pt] at (Uright) {$|x|=m+K$};
\node[labelstyle, right=6pt] at (Mright) {$|x|=m$};
\node[labelstyle, right=6pt] at (Lright) {$|x|=m-K$};

\node[labelstyle, blue!55!black] at (0,2.55) {$H_{m+K}\subseteq A_n$};
\node[labelstyle, blue!55!black] at (-0.2,0.95) {$A_n$};
\node[labelstyle, orange!60!black] at (0,-0.95) {$H_{m-K}$};



\end{tikzpicture}
\caption{Illustration of a random monotone subset $A_n\subseteq \Omega_n$}
\label{fig:RandomMonotoneSet}
\end{figure}

\section{Log-Sobolev inequality for Hamming caps}\label{sec:LSI-caps}

The goal of this section is to prove the following estimate. 

\begin{theorem}[Sharp LSI for Hamming caps]\label{thm:hamming-caps}
For every $0\leq s< n$, 
\[
 t_{\operatorname{LS}}(H_s)  \asymp
        n\log\frac{en}{n-s+1}.
\]
\end{theorem}

We write $\pi_s$ for the uniform measure on $H_s$, and $\Ecal_s$ for the Dirichlet form of the censored walk on $H_s$. Let $p_k=\pi_s(E_k)$ and $\nu_k$ be the uniform measure on $E_k$. 

Throughout the section, we denote $X_k\sim \nu_k$, and we write $X_k^{\uparrow}$ (resp. $X_k^{\downarrow}$) for the random point obtained from $X_k$ by flipping a uniformly chosen zero-coordinate (resp. one-coordinate) to one (resp. zero). More precisely, conditionally on $X_k=x$, we have 
$$
\mathbb{P}(X_k^\uparrow=y\mid X_k=x)
=\frac{1}{n-k}\mathbf 1_{\{y\in E_{k+1},\, y\sim x\}}.
$$
Notice also that $X_k^{\uparrow}\sim\nu_{k+1}$ and that $(X_k,X_k^{\uparrow})$ has the same law as $(X_{k+1}^{\downarrow}, X_{k+1})$. By reversibility, every edge of the censored walk inside $H_s$ may be counted once, oriented from $E_k$ to $E_{k+1}$. Hence
\begin{align}
\Ecal_{H_s}(f,f)&= \frac{1}{2n|H_s|}
\sum_{k=s}^{n-1} \sum_{\substack{x\in E_k,\, y\in E_{k+1}\\ x\sim y}} (f(y)-f(x))^2  \nonumber \\
&=\sum_{k=s}^{n-1} p_k\frac{n-k}{2n}
\mathbb{E}{
\bigl(f(X_k^\uparrow)-f(X_k)\bigr)^2}\label{eq:Dir-up-walk},
\end{align}

Our starting point is the entropy decomposition with respect to the partition $H_s=\cup_{k\geq s}E_{k}$, which separates a radial contribution corresponding to the distribution across the Hamming levels, from a tangential contribution corresponding to the oscillation within each level.  

\begin{lemma}[Entropy decomposition]\label{lem:entropy-decomp}
For every $0\le s \le n$ and $f:\, H_s\to \R$, we have
$$     
\Ent_{\pi_s}(f^2) =\Ent_{p}(\bar{f}^2)+\sum_{k= s}^{n-1}p_{k}\Ent_{\nu_k}(f^2),
$$
where $\bar{f}:\, \{s,\ldots,n\}\to \R$ is defined by $\bar{f}(k)= (\mathbb{E}_{\nu_k}f^2)^{\frac12}=\Vert f(X_k)\Vert_{L^2}$. 
\end{lemma}

The radial contribution corresponds to a log-Sobolev inequality for a birth-death chain, which is well studied and follows, for instance, from the Hardy-type criteria for birth-death chains (see~\cite{Miclo}). In the present setting, however, the needed bound follows immediately from the log-Sobolev inequality on the full cube. We include the short argument for completeness.

\begin{proposition}[Radial contribution]
\label{prop:radial-contribution}
For every $0\leq s\le n-1$ and $f:\, H_s\to \R$, 
$$
\Ent_{p}(\bar{f}^2)\leq  2\,n\, \Ecal_{H_s}(f,f).
$$
\end{proposition}
\begin{proof}
Given $0\leq s\le n-1$ and $u:\, \{s,\ldots, n\}\to \R$, we define its radial extension $U:\, \{0,1\}^n\to \R$ by $U(x)= u(s\vee |x|)$. Note that $U$ is constant on $H_s^c$ and $U(x)=u(|x|)$ on $H_s$. Let $X\sim \pi_s$. Since the law of $|X|$ under $\pi_s$ is $p=(p_s,\ldots,p_n)$, we have 
$$
\Ent_{\pi_s}(U^2\vert_{H_s})= \Ent_{p}(u^2)
$$
This, together with the entropy decomposition with respect to the partition $\Omega_n=H_s\cup H_s^c$, implies that  
$$
\Ent_{\pi}(U^2)\geq \pi(H_s)\Ent_{\pi_s}(U^2\vert_{H_s})= \pi(H_s) \Ent_{p}(u^2).
$$
Using the log-Sobolev inequality $\Ent_{\pi}(U^2)\leq 2n\Ecal_{\Omega_n}(U, U)$ for the full cube~\cite[Theorem 5.1]{BLM} and observing that 
$$
\Ecal_{\Omega_n}(U,U)= \pi(H_s) \sum_{k=s}^{n-1} p_k \frac{n-k}{2n} \big(u(k+1)-u(k)\big)^2,
$$
we deduce 
$$
\Ent_{p}(u^2)\leq 2\,n\,\sum_{k=s}^{n-1} p_k \frac{n-k}{2n} \big(u(k+1)-u(k)\big)^2. 
$$
Applying this to $u=\bar{f}$, and using the triangle inequality 
$$
|\bar{f}(k+1)-\bar{f}(k)|\leq  \Vert f(X_k^{\uparrow})-f(X_k)\Vert_{L^2},
$$
the result follows in view of \eqref{eq:Dir-up-walk}.

\end{proof}
We now turn to the entropy inside the Hamming levels. The natural auxiliary chain on $E_k$ is the Bernoulli-Laplace chain \cite[Section XV.2]{Feller1}: from $x\in E_k$, pick two coordinates $i$ and $j$ at random and swap their values, that is, we define a new $x^{(ij)}\in \Omega_n$ by swapping the $i$-th and the $j$-th coordinates of $x$ and keeping the rest of the coordinates unchanged. 

Following~\cite{LeeYau}, we define the Dirichlet energy of the Bernoulli-Laplace chain on $E_k$ as
\begin{align*}
    \Ecal^{\operatorname{BL}}_{k}(g,g):=\frac{1}{2n}\mathbb{E}\left[\sum_{i,j: i<j}\left(g(X_k)-g(X^{(ij)}_k)\right)^2\right]
\end{align*}
We rewrite the Dirichlet energy in a more convenient form. Given $X_k=x\in E_{k}$, let $(u, v)$ be such that $u$ is a uniformly random zero coordinate in $x$ and $v$ is a uniformly random one coordinate in $x$. Then, 
\begin{align}
\label{eqn:BL-Dir}
    \Ecal^{\operatorname{BL}}_{k}(g,g):=\frac{k(n-k)}{2n}\mathbb{E}\left[\left(g(X_k)-g(X^{(uv)}_k)\right)^2\right]\;,
\end{align}
where the expectation is with respect to the tuple $(X_k,u,v)$ where $X_k\sim \nu_k$.

\begin{lemma}[Up-down representation of Bernoulli-Laplace]\label{lem:up-down-BL}
   Let $1\leq k\leq n-1$. Then, for every $g:\, \Omega_n\to \R$, we have 
   $$
   \Ecal^{\operatorname{BL}}_k(g,g)\leq \frac{(k+1)(n-k)}{n} \mathbb{E}\left[(g(X^{\uparrow}_k)-g(X_k))^2\right], 
   $$
   and 
    $$
   \Ecal^{\operatorname{BL}}_k(g,g)\leq \frac{k(n-k+1)}{n} \mathbb{E}\left[(g(X^{\downarrow}_k)-g(X_k))^2\right]. 
   $$
\end{lemma}

\begin{proof}
We prove the upward estimate, since the downward one is identical upon exchanging zeros and ones. 
Given $X_k=x\in E_k$, let $I$ be a uniform random zero coordinate in $x$ and let $X^{\uparrow}_k=X_k+e_{I}\in E_{k+1}$. Let $J$ be a uniform random one coordinate in $X_k^\uparrow$ (independent of $I$), and let $Z_k=(X_k^\uparrow)^{\downarrow}=X_k^\uparrow-e_J$. Note that $X_k\ne Z_k$ with probability $k/(k+1)$. Furthermore, given $X_k\neq Z_k$, $(I, J)$ has the same distribution as $(u, v)$. In particular, 
\[
\mathbb{E}\left[\left(g(X_k)-g(Z_k)\right)^2\right] = \frac{k}{k+1}\mathbb{E}\left[\left(g(X_k)-g(X^{(uv)}_k)\right)^2\right]\;.
\]
Combining with~\eqref{eqn:BL-Dir}, we get
\[
    \mathcal{E}_k^{BL}(g, g) = \frac{(n-k)(k+1)}{2n} \mathbb{E}\left[\left(g(X_k)-g(Z_k)\right)^2\right]\;.
\]
We now observe that $X_k$ and $Z_k$ are independent given $X^\uparrow_k$, as $I$ and $J$ are independent. Furthermore, given $X_k^\uparrow$, $X_k$ and $Z_k$ have the same distribution as the distribution of $(X^\uparrow_k)^\downarrow$. Therefore,
    \[
    \frac{1}{2}\mathbb{E}\left[\left(g(X_k)-g(Z_k)\right)^2\vert X^\uparrow_k \right]=\var(g(X_k)\vert X_k^\uparrow)\leq \mathbb{E}\left[\left(g(X_k)-g(X_k^\uparrow)\right)^2\Big| X_k^\uparrow\right]\;.
    \]
The result follows by taking the expectation.
\end{proof}

\begin{proposition}[Tangential contribution]\label{prop:tangential-contribution}
    For every $1\leq s\leq n-1$ and $f:\, H_s\to \R$,
    $$
    \sum_{k= s}^{n-1} p_k\Ent_{\nu_k}(f^2)\leq 36\,n \log \frac{en}{n-s}\, \Ecal_{H_s}(f,f). 
    $$
\end{proposition}
\begin{proof}
For every $1\leq k\leq n-1$, we make use of the sharp log-Sobolev inequality for Bernoulli-Laplace (see \cite[Theorem~5]{LeeYau})
    $$
    \Ent_{\nu_k}(f^2)\leq \frac{2}{\log 2} \,\log \frac{n^2}{k(n-k)} \Ecal_{k}^{BL}(f,f).
    $$
   Combined with Lemma~\ref{lem:up-down-BL}, this implies   
   \begin{equation}\label{eq:entropy-upwalk}
       \Ent_{\nu_k}(f^2)\leq \frac{2}{\log 2} \, \frac{(k+1)(n-k)}{n}\log \frac{n^2}{k(n-k)} \mathbb{E}{(f(X^{\uparrow}_k)-f(X_k))^2}\;,
   \end{equation}
    and 
      \begin{equation}\label{eq:entropy-downwalk}
    \Ent_{\nu_k}(f^2)\leq \frac{2}{\log 2} \, \frac{k(n-k+1)}{n}\log \frac{n^2}{k(n-k)} \mathbb{E}{(f(X^{\downarrow}_k)-f(X_k))^2}.
   \end{equation}
Denote, for brevity, 
\begin{align*}
    \Gamma_k=p_k\frac{n-k}{n}\mathbb{E}{\bigl(f(X^\uparrow_k)-f(X_k)\bigr)^2}.
\end{align*}
Using that $p_k=\frac{n-k+1}{k} p_{k-1}$ and $(X_k^{\downarrow}, X_k)\sim (X_{k-1},X_{k-1}^{\uparrow})$, the inequalities \eqref{eq:entropy-upwalk} and \eqref{eq:entropy-downwalk} reduce to 
$$
p_k\Ent_{\nu_k}(f^2)\leq \frac{2}{\log 2} \log \frac{n^2}{k(n-k)} \, \min\big((k+1)\Gamma_{k}, (n-k+1)\Gamma_{k-1}\big).
$$
Using the elementary relation $\min (k+1, n-k+1)\, \log \frac{n^2}{k(n-k)}\leq 2n$, we get that 
$$
\begin{cases}
    p_k\Ent_{\nu_k}(f^2)\leq \frac{4n}{\log 2} \Gamma_{k}  & \text{ if } k\leq n-k;\\
    & \\
     p_k\Ent_{\nu_k}(f^2)\leq \frac{4n}{\log 2} \Gamma_{k-1}  & \text{ otherwise.}
\end{cases}
$$
Summing over $k> s$ and using \eqref{eq:Dir-up-walk}, we get 
$$
\sum_{k=s+1}^{n-1} p_k\Ent_{\nu_k}(f^2)\leq \frac{16\,n}{\log 2}\, \Ecal_{H_s}(f,f).
$$
Finally, we use \eqref{eq:entropy-upwalk} with $k=s$ to get
$$
p_s\Ent_{\nu_s}(f^2)\leq \frac{2(s+1)}{\log 2} \log \frac{n^2}{s(n-s)} \, \Ecal_{H_s}(f,f)\leq \frac{8\,n}{\log 2} \log \frac{en}{n-s}\, \Ecal_{H_s}(f,f). 
$$
Combining the above and using that $\frac{24}{\log 2}\leq 36$, we finish the proof. 
\end{proof}

\begin{proof}[Proof of Theorem~\ref{thm:hamming-caps}]
The case $s=0$ corresponds to the full cube. Therefore, we assume $1\leq s\leq n-1$. The upper bound readily follows by combining Lemma~\ref{lem:entropy-decomp} with Propositions~\ref{prop:radial-contribution} and \ref{prop:tangential-contribution}.

To get the lower bound, fix $z\in E_s$ and take $f=\mathbf 1_{\{z\}}$. Then, we have 
$$
\Ent_{\pi_s}(f^2)=\frac{\log |H_s|}{|H_s|}. 
$$
On the other hand, $z$ has exactly $n-s$ neighbors inside $H_s$. Therefore, 
$$
\Ecal_{H_s}(f,f)= \frac{n-s}{2n |H_s|}. 
$$
Combining the two identities, we get 
$$
\frac{\Ent_{\pi_s}(f^2)}{\Ecal_{H_s}(f,f)}= \frac{2n\log |H_s|}{n-s}.
$$
It remains to use that $|H_s|\geq \binom{n}{s}$ if $s\geq n/2$ and that $|H_s|\geq 2^{n-1}$ when $s<n/2$. 
\end{proof}

\section{The log-Sobolev inequality for sandwiched monotone sets}\label{sec:LSI-sandwich}

The goal of this section is to transfer the logarithmic Sobolev inequality from Hamming caps to monotone sets lying between two nearby caps. The main ingredient is a one-layer extension lemma: whenever a Hamming layer is added, functions can be extended harmonically to the new vertices while increasing the Dirichlet energy by at most a constant factor. Iterating this extension across the layers of the sandwich yields a comparison theorem between the logarithmic Sobolev constants of the monotone set and the surrounding cap.

\begin{lemma}[One layer harmonic extension]\label{lem:one-step-extension}
    Let $A\subseteq \Omega_n$ be a monotone subset, and for every $0\leq l\leq n-1$, set $D_l=A\cup H_l$. For every $0\leq l\leq n-1$ and every function $f:\, D_{l+1}\to \R$, there exists an extension $F:\, D_{l}\to \R$ such that $F_{\mid D_{l+1}}=f$ and 
    $$
    \Ecal_{D_l}(F,F)\leq \frac{|D_{l+1}|}{|D_l|} \Big(1+\frac{l+2}{n-l}\Big)\, \Ecal_{D_{l+1}}(f,f). 
    $$
\end{lemma}
\begin{proof}
 Fix $0\leq l\leq n-1$, and note that $D_l\setminus D_{l+1}=E_l\setminus A$. For $x\in E_l\setminus A$, let 
 $$
 U(x)=\{y\in E_{l+1}:\, y\sim x\}
 $$
 be the set of upper neighbors of $x$. Since $H_{l+1}\subseteq D_{l+1}$, all these neighbors belong to $D_{l+1}$. 
 Let $f:\, D_{l+1}\to \R$ and define 
 $$
 F(x)=\frac{1}{n-l}\sum_{y\in U(x)} f(y),\quad x\in E_l\setminus A,
 $$
 and set $F=f$ on $D_{l+1}$. 

Let $x\in E_l\setminus A$. We claim that every neighbor of $x$ in $D_{l+1}$ belongs to $U(x)$. Indeed, let $y\in E_{l-1}$ be a neighbor of $x$. Then $y\not\in H_{l+1}$, and if $y\in A$, monotonicity of $A$ would imply that $x\in A$, a contradiction. Therefore $y\not\in D_{l+1}$. 

Thus, the only new edges created when passing from $D_{l+1}$ to $D_l$ are the edges joining $x\in E_l\setminus A$ to its upper neighbors in $U(x)$. Therefore 
\begin{equation}
\label{eq:decomp-Dirichlet}
    \Ecal_{D_{l}}(F, F) = \frac{|D_{l+1}|}{|D_{l}|}\Ecal_{D_{l+1}}(f, f) + \frac{1}{2n|D_{l}|}\sum_{x\in E_l\setminus A}\sum_{y\in U(x)}(F(x)-f(y))^2\;.
\end{equation}
Since $F(x)$ is the average of $f$ on $U(x)$, the variance identity yields 
$$
    \sum_{y\in U(x)}(F(x)-f(y))^2= \frac{1}{2(n-l)} \sum_{y,z\in U(x)} (f(y)-f(z))^2. 
$$
Summing over $x$, and using that every pair $y,z\in E_{l+1}$ with $|y-z|=2$ has at most one common lower neighbor in $E_l$, we get 
\begin{equation}\label{eq:after-variance-identity}
  \sum_{x\in E_l\setminus A}\sum_{y\in U(x)}(F(x)-f(y))^2\leq \frac{1}{2(n-l)} \sum_{\substack{y,z\in E_{l+1}\\ |y-z|=2}}  (f(y)-f(z))^2.
\end{equation}
For such a pair $y,z$, let $w=y\vee z\in E_{l+2}$. Then $w\in H_{l+1}\subseteq D_{l+1}$, and $w$ is adjacent to both $y$ and $z$. Let $L(w)=\{y\in E_{\ell+1}: y\sim w\}$. Write $\bar{f}_{L(w)}=\frac{1}{|L(w)|}\sum_{y\in L(w)}f(y)$ for the average of $f$ over $L(w)$. Using the variance identity over the set $L(w)$, we get 
\[
\frac{1}{2}\sum_{y, z\in L(w)}(f(y)-f(z))^2 = |L(w)|\sum_{y\in L(w)}(f(y)-\bar{f}_{L(w)})^2\leq |L(w)|\sum_{y\in L(w)}(f(y)-f(w))^2\;.
\]
Using this together with \eqref{eq:after-variance-identity} and the fact that $|L(w)|=l+2$, we get 
$$
\sum_{x\in E_l\setminus A}\sum_{y\in U(x)}(F(x)-f(y))^2\leq \frac{(l+2)}{(n-l)} \sum_{\substack{y,w\in D_{l+1}\\ y\sim w}}  (f(y)-f(w))^2.
$$
Plugging this back in \eqref{eq:decomp-Dirichlet}, we finish the proof. 
\end{proof}

\begin{theorem}[Comparison for sandwiched monotone sets]\label{thm:sandwich-monotone-set}
    Let $0\leq s\leq r\leq n-1$, and let $A\subseteq \Omega_n$ be a monotone set satisfying 
    $$
    H_r\subseteq A\subseteq H_s. 
    $$
Then 
$$
 t_{\operatorname{LS}}(A)\leq t_{\operatorname{LS}}(H_s)\, \prod_{l=s}^{r-1} \left(1+\frac{l+2}{n-l}\right).
$$
\end{theorem}
\begin{proof}
    For every $s\leq l\leq r$, set $D_l=A\cup H_l$. Observe that $D_s=H_s$ while $D_r=A$. Let $f:\, A\to \R$. Applying Lemma~\ref{lem:one-step-extension} successively for $l=s,\ldots, r-1$ , we obtain an extension $F:\, H_s\to \R$ of $f$ satisfying 
    \begin{equation}\label{eq:dirichlet-comparison}
        \Ecal_s(F,F)\leq \frac{|A|}{|H_s|}\prod_{l=s}^{r-1} \left(1+\frac{l+2}{n-l}\right)\, \Ecal_A(f,f). 
    \end{equation}
We now compare the entropies. Since $F_{\mid A}=f$, the entropy decomposition associated with the partition $H_s=A\cup (H_s\setminus A)$ gives 
\begin{equation}\label{eq:entropy-comparison}
   \Ent_{\pi_s}(F^2)\geq \frac{|A|}{|H_s|} \Ent_{\pi_A}(f^2). 
\end{equation}
Applying the log-Sobolev inequality on $H_s$, and combining it with \eqref{eq:dirichlet-comparison} and \eqref{eq:entropy-comparison}, we get
$$
\Ent_{\pi_A}(f^2)\leq t_{\operatorname{LS}}(H_s)  \prod_{l=s}^{r-1} \left(1+\frac{l+2}{n-l}\right)\, \Ecal_A(f,f).
$$
The statement follows. 
\end{proof}

\begin{proof}[Proof of Theorem~\ref{thm:highprobLSI}]
    Let $A_n$ be a uniformly chosen monotone subset of $\Omega_n$. By \eqref{eqn:TypicalA}, there exists a universal constant $K$ such that 
$$
    \mathbb{P}(H_{m+K}^{(n)}\subseteq A_n\subseteq H_{m-K}^{(n)})=1-o_n(1)\;.
$$
On this event, Theorem~\ref{thm:sandwich-monotone-set} together with Theorem~\ref{thm:hamming-caps} yields
$$
 t_{\operatorname{LS}}(A_n)\leq C_K  t_{\operatorname{LS}}(H_{m-K})\leq C_K' n,
$$
for some constants $C_K, C_K'$ depending only on $K$. 
To prove the matching lower bound, note that on the same event we have $H_{m+K}\subseteq A_n$ and therefore 
$$
\operatorname{diam}(A)\ge \operatorname{diam}(H_{m+K})\ge \frac{n}{2}-K.
$$
It remains to use a diameter lower bound for the logarithmic Sobolev constant (see \cite[Lemma 1]{SY25}) to conclude the proof. 
\end{proof}

\section{Monotone sets with large log-Sobolev constant}\label{sec:large-LSI}

In this section, we investigate monotone sets with large log-Sobolev constants. We show that the largest possible order is $n^2$ and that this order is attained by monotone sets that occupy at least a constant fraction of the hypercube. The lower bound arises from a simple bottleneck mechanism. Let $A\subseteq \Omega_n$ and $U\subseteq A$, and define the \emph{edge boundary} as 
\[
\partial_{A}U :=\{\{x, y\}: x\in U, y\in A\setminus U, x\sim y\}\;.
\] 
Applying the logarithmic Sobolev inequality to the indicator function $\mathbf{1}_U$ yields 
\begin{align}\label{equation: universal lower bound}
    t_{\operatorname{LS}}(A)\geq 2n\cdot \frac{|U|\,\log (|A|/|U|)}{|\partial_AU|}\;.
\end{align}
The lower bound depends not only on the size of the boundary $\partial_AU$, but also on the relative size of $U$ inside $A$. Consequently, a subset $U$ may force a large logarithmic Sobolev constant even when its boundary is comparable to its volume, provided that $U$ occupies only an exponentially small fraction of $A$. We end this section with a simple example to illustrate this. 

\begin{proposition}[Worst-case behavior at fixed density]\label{prop:wortcase-LSI}
    Let $\alpha\in (0,1/2]$ be fixed. Then 
    $$
     \sup_{\substack{A\subseteq \Omega_n \text{ monotone}\\ |A|\geq \alpha |\Omega_n|} }t_{\operatorname{LS}}(A)  \asymp_\alpha n^2
    $$
\end{proposition}

\begin{proof}
We begin with the upper bound. Let $A\subseteq \Omega_n$ be a monotone set $|A|\geq \alpha|\Omega_n|$. By the log-Sobolev inequality for the rank-one chain \cite[Theorem 2.2.9]{S-C97}, there exists a universal constant $C$ such that
   $$
   \Ent_{\pi_A}(f^2)\leq C\log  |A|\, \var_{\pi_A}(f),
   $$
   for any function $f:\, A\to \R$.
   
On the other hand, the sharp Poincaré inequality for monotone sets \cite{fei2025spectral} yields 
$$
\var_{\pi_A}(f)\leq C' n \Ecal_A(f,f),
$$
for some constant $C'$ depending only on $\alpha$. Combining the two estimates and using the fact that $|A|\leq |\Omega_n|=2^n$ yields the upper bound. 

For the reverse inequality, fix $q=q(\alpha)\geq 1$ satisfying $1-2^{-q}\geq \alpha$. Write $n=m+q$ and define 
$$
A_{n,q}= \{(u,v)\in \{0,1\}^{m}\times \{0,1\}^q:\, v\neq 0\}\cup \{z\},
$$
where $z$ is the vector whose first $m$ coordinates are equal to one and its last $q$ are equal to zero. 
Clearly, $A_{n,q}$ is a monotone set and 
$|A_{n,q}|= (1-2^{-q})2^n+1\geq \alpha |\Omega_n|$.
Note that $z$ has exactly $q$ neighbors in $A_{n,q}$. Therefore, setting $U=\{z\}$ in \eqref{equation: universal lower bound} , we get
\begin{align*}
    t_{\operatorname{LS}}(A_{n,q})\geq \frac{2n \log |A_{n,q}|}{q}.
\end{align*}
The proof is complete once the lower bound on $|A_{n,q}|$ is used.
\end{proof}

\begin{remark}
    Even though the examples in Proposition~\ref{prop:wortcase-LSI} have a logarithmic Sobolev constant of order $n^2$, they do not contradict the conjecture of Ding and Mossel. Indeed, the obstruction comes from the exceptional vertex $z$, whose mass is exponentially small. While the presence of $z$ forces $t_{\operatorname{LS}}(A)  \asymp_\alpha n^2$, the mean time to leave $z$ is of order $n$. After leaving $z$, the walk behaves essentially like a product chain, suggesting the mixing time remains of order $n\log n$. 

    The same example also shows that the modified logarithmic Sobolev inequality \cite{bobkov2006modified} cannot hold with a constant of order $n$; in fact, the above example provides a lower bound of order $n^2/\log n$. On the other hand, the arguments in Section~\ref{sec:LSI-caps} and \ref{sec:LSI-sandwich} extend with minor modifications, yielding order $n$ modified logarithmic Sobolev constant for all Hamming caps $H_s$, $1\leq s\leq n-1$, and for a random monotone set. 
\end{remark}

\section*{Acknowledgment.}
This work is supported in part by the NYUAD Center for Interdisciplinary Data Science \& AI (CIDSAI), funded by Tamkeen under the NYUAD Research Institute Award CG016.

\bibliographystyle{alpha} 
\bibliography{ref}

\newcommand{\etalchar}[1]{$^{#1}$}
\begin{thebibliography}{GGL{\etalchar{+}}00}

\bibitem[Bec75]{Beck}
William Beckner.
\newblock Inequalities in {F}ourier analysis.
\newblock {\em Ann. of Math. (2)}, 102(1):159--182, 1975.

\bibitem[BLM13]{BLM}
St\'ephane Boucheron, G\'abor Lugosi, and Pascal Massart.
\newblock {\em Concentration inequalities}.
\newblock Oxford University Press, Oxford, 2013.

\bibitem[Bol01]{Bollobas}
B\'ela Bollob\'as.
\newblock {\em Random graphs}, volume~73 of {\em Cambridge Studies in Advanced Mathematics}.
\newblock Cambridge University Press, Cambridge, {S}econd edition, 2001.

\bibitem[Bon70]{Bon}
Aline Bonami.
\newblock \'{E}tude des coefficients de {F}ourier des fonctions de {$L\sp{p}(G)$}.
\newblock {\em Ann. Inst. Fourier (Grenoble)}, 20:335--402, 1970.

\bibitem[BR06]{BR06}
B\'ela Bollob\'as and Oliver Riordan.
\newblock {\em Percolation}.
\newblock Cambridge University Press, New York, 2006.

\bibitem[BT06]{bobkov2006modified}
Sergey~G. Bobkov and Prasad Tetali.
\newblock Modified logarithmic {S}obolev inequalities in discrete settings.
\newblock {\em J. Theoret. Probab.}, 19(2):289--336, 2006.

\bibitem[CSY25]{chang2025poincar}
Fan Chang, Guowei Sun, and Lei Yu.
\newblock Poincar\'e inequality on the monotone sets revisited.
\newblock {\em arXiv preprint arXiv:2506.09852}, 2025.

\bibitem[DM14]{ding2014mixing}
Jian Ding and Elchanan Mossel.
\newblock Mixing under monotone censoring.
\newblock {\em Electron. Commun. Probab.}, 19:1--6, 2014.

\bibitem[DSC96]{Dia96LSI}
P.~Diaconis and L.~Saloff-Coste.
\newblock Logarithmic {S}obolev inequalities for finite {M}arkov chains.
\newblock {\em Ann. Appl. Probab.}, 6(3):695--750, 1996.

\bibitem[Fat19]{fathi2019quelques}
Max Fathi.
\newblock {\em Quelques applications du transport optimal en analyse et en probabilit{\'e}s}.
\newblock PhD thesis, Universit{\'e} Paul Sabatier (Toulouse 3), 2019.

\bibitem[Fel68]{Feller1}
William Feller.
\newblock {\em An introduction to probability theory and its applications. {V}ol. {I}}.
\newblock John Wiley \& Sons Inc., New York-London-Sydney, {T}hird edition, 1968.

\bibitem[FFP25]{fei2025spectral}
Yumou Fei and Renato Ferreira~Pinto, Jr.
\newblock On the spectral expansion of monotone subsets of the hypercube.
\newblock In {\em Approximation, randomization, and combinatorial optimization. {A}lgorithms and techniques}, volume 353 of {\em LIPIcs. Leibniz Int. Proc. Inform.}, pages Art. No. 42, 24. Schloss Dagstuhl. Leibniz-Zent. Inform., Wadern, 2025.

\bibitem[FK96]{FK96}
Ehud Friedgut and Gil Kalai.
\newblock Every monotone graph property has a sharp threshold.
\newblock {\em Proc. Amer. Math. Soc.}, 124(10):2993--3002, 1996.

\bibitem[FK13]{FK13}
James~Allen Fill and Jonas Kahn.
\newblock Comparison inequalities and fastest-mixing {M}arkov chains.
\newblock {\em Ann. Appl. Probab.}, 23(5):1778--1816, 2013.

\bibitem[FKG71]{FKG}
C.~M. Fortuin, P.~W. Kasteleyn, and J.~Ginibre.
\newblock Correlation inequalities on some partially ordered sets.
\newblock {\em Comm. Math. Phys.}, 22:89--103, 1971.

\bibitem[GGL{\etalchar{+}}00]{GGLRS}
Oded Goldreich, Shafi Goldwasser, Eric Lehman, Dana Ron, and Alex Samorodnitsky.
\newblock Testing monotonicity.
\newblock {\em Combinatorica}, 20(3):301--337, 2000.

\bibitem[Gri89]{Grim89}
Geoffrey Grimmett.
\newblock {\em Percolation}.
\newblock Springer-Verlag, New York, 1989.

\bibitem[Gri06]{G06}
Geoffrey Grimmett.
\newblock {\em The random-cluster model}, volume 333 of {\em Grundlehren der mathematischen Wissenschaften [Fundamental Principles of Mathematical Sciences]}.
\newblock Springer-Verlag, Berlin, 2006.

\bibitem[Gro75]{Gross}
Leonard Gross.
\newblock Logarithmic {S}obolev inequalities.
\newblock {\em Amer. J. Math.}, 97(4):1061--1083, 1975.

\bibitem[Hol11]{Hol11}
Alexander~E. Holroyd.
\newblock Some circumstances where extra updates can delay mixing.
\newblock {\em J. Stat. Phys.}, 145(6):1649--1652, 2011.

\bibitem[KKL88]{KKL88}
Jeff Kahn, Gil Kalai, and Nathan Linial.
\newblock The influence of variables on {B}oolean functions.
\newblock In {\em 29th {A}nnual {S}ymposium on {F}oundations of {C}omputer {S}cience}, pages 68--80. IEEE Comput. Soc. Press, Washington, DC, 1988.

\bibitem[Kor81]{Korshunov81}
A.~D. Korshunov.
\newblock The number of monotone {B}oolean functions.
\newblock {\em Problemy Kibernet.}, (38):5--108, 272, 1981.

\bibitem[Kor03]{Korshunov03}
A.~D. Korshunov.
\newblock Monotone {B}oolean functions.
\newblock {\em Uspekhi Mat. Nauk}, 58(5(353)):89--162, 2003.

\bibitem[KS06]{KS06}
Gil Kalai and Shmuel Safra.
\newblock Threshold phenomena and influence: perspectives from mathematics, computer science, and economics.
\newblock In {\em Computational complexity and statistical physics}, St. Fe Inst. Stud. Sci. Complex., pages 25--60. Oxford Univ. Press, New York, 2006.

\bibitem[Led99]{Led99Concentration}
Michel Ledoux.
\newblock Concentration of measure and logarithmic {S}obolev inequalities.
\newblock In {\em S\'eminaire de {P}robabilit\'es, {XXXIII}}, volume 1709 of {\em Lecture Notes in Math.}, pages 120--216. Springer, Berlin, 1999.

\bibitem[LY98]{LeeYau}
Tzong-Yow Lee and Horng-Tzer Yau.
\newblock Logarithmic {S}obolev inequality for some models of random walks.
\newblock {\em Ann. Probab.}, 26(4):1855--1873, 1998.

\bibitem[Mic99]{Miclo}
L.~Miclo.
\newblock An example of application of discrete {H}ardy's inequalities.
\newblock {\em Markov Process. Related Fields}, 5(3):319--330, 1999.

\bibitem[O'D14]{ODon14}
Ryan O'Donnell.
\newblock {\em Analysis of {B}oolean functions}.
\newblock Cambridge University Press, New York, 2014.

\bibitem[Sal25]{salez2025modern}
Justin Salez.
\newblock Modern aspects of markov chains: entropy, curvature and the cutoff phenomenon.
\newblock {\em arXiv preprint arXiv:2508.21055}, 2025.

\bibitem[SC97]{S-C97}
Laurent Saloff-Coste.
\newblock Lectures on finite {M}arkov chains.
\newblock In {\em Lectures on probability theory and statistics ({S}aint-{F}lour, 1996)}, volume 1665 of {\em Lecture Notes in Math.}, pages 301--413. Springer, Berlin, 1997.

\bibitem[SY25]{SY25}
Justin Salez and Pierre Youssef.
\newblock Intrinsic regularity in the discrete log-{S}obolev inequality, 2025.
\newblock To appear in the Journal of European Mathematical Society.

\end{thebibliography}

\end{document}